\documentclass{amsart}
\usepackage{amsmath}

  \usepackage{paralist}
  \usepackage{graphics} %% add this and next lines if pictures should be in esp format
  \usepackage{epsfig} %For pictures: screened artwork should be set up with an 85 or 100 line screen
 \usepackage[colorlinks=true]{hyperref}
\hypersetup{urlcolor=blue, citecolor=red}

\newtheorem{theorem}{Theorem}[section]
\newtheorem{corollary}{Corollary}

\newtheorem{lemma}[theorem]{Lemma}
\newtheorem{proposition}{Proposition}

\theoremstyle{definition}
\newtheorem{definition}[theorem]{Definition}
\newtheorem{remark}{Remark}

\title[Regularity in Reifenberg flat domains]
      {Boundary regularity for the Poisson equation in reifenberg-flat domains}

\author[Antoine Lemenant and Yannick Sire]{}

\subjclass{Primary:  ; Secondary: .}

 \email{sire@cmi.univ-mrs.fr}
 \email{lemenant@ljll.univ-paris-diderot.fr}

\newcommand{\R}{\mathbb R}

\begin{document}

%The abstract of your paper

\maketitle
% Enter the first author's name and address:

\centerline{\scshape Antoine Lemenant}
\medskip
{\footnotesize
 % please put the address of the second  and third author
 \centerline{ LJLL, Universit\'e Paris-Diderot, CNRS }
 %\centerline{ Universit\'e Paris Diderot - Paris 7 }
 %  \centerline{U.F.R de Math\'ematiques, Site Chevaleret}
 %  \centerline{Case 7012, 75205 Paris Cedex 13, France}
 \centerline{Paris, France}
}

\medskip

\centerline{\scshape Yannick Sire }
\medskip
{\footnotesize
% please put the address of the first author
 \centerline{LATP, Universit\'e Aix-Marseille, CNRS}
   %\centerline{Other lines}
   \centerline{Marseille, France}
} % 

\bigskip

\begin{abstract} This paper is devoted to the investigation of the boundary regularity for the Poisson equation 
$$
\left\{
\begin{array}{cc}
-\Delta u = f & \text{ in } \Omega \\
u= 0 & \text{ on } \partial \Omega
\end{array}
\right.
$$
where $f$ belongs to some $L^p(\Omega)$ and $\Omega$ is a Reifenberg-flat domain of $\mathbb R^n.$ More precisely, we prove that given an exponent $\alpha\in (0,1)$, there exists an $\varepsilon>0$ such that the solution $u$ to the previous system is locally H\"older continuous provided that $\Omega$ is $(\varepsilon,r_0)$-Reifenberg-flat. The proof is based on  Alt-Caffarelli-Friedman's monotonicity formula and  Morrey-Campanato theorem. 
\end{abstract}

The goal of the present paper is to prove a boundary regularity result for the Poisson equation with homogeneous  Dirichlet boundary conditions in non smooth domains. We consider the case of Reifenberg-flat domains as given by the following definition

\begin{definition}\label{defreif} Let $\varepsilon, r_0$ be two real numbers satisfying $0 < \varepsilon<1/2$ and $r_0 >0$. An $(\varepsilon,r_0)$-Reifenberg-flat domain $\Omega \subset \R^N$ is an open,  bounded, and connected  set satisfying the two following  conditions  :

\noindent (i) for every $x \in \partial \Omega$ and for any $r\leq r_0$,  there exists a hyperplane $P(x,r)$ containing $x$ which satisfies
\begin{eqnarray*}
\frac{1}{r}d_H( \partial \Omega \cap B(x,r), P(x,r)\cap B(x,r))\leq \varepsilon. \label{reif}
\end{eqnarray*}

\noindent (ii)  for every $x \in \partial \Omega$,  one of the connected component of  
$$B(x,r_0)\cap \{x ; d(x,P(x,r_0))\geq 2\varepsilon r_0\}$$ 
is contained in   $\Omega$ and the other one is contained in $\Omega^c$.
\end{definition}

We consider the following problem in the $(\varepsilon, r_0)$-Reifenberg-flat domain
$\Omega \subset \R^N$ for some $f \in L^q(\Omega)$,
$$
(P1)\left\{
\begin{array}{cc}
-\Delta u = f & \text{ in } \Omega \\
u= 0 & \text{ on } \partial \Omega
\end{array}
\right.
$$

Reifenberg-flat domains are less smooth than Lipschitz domains and it is well known that we cannot expect more regularity than H\"older for boundary regularity of the Poisson equation in Lipschitz domains (see \cite{grisvard}, \cite[Remark 17]{lms}, or \cite{jklip}).

Historically, Reifenberg-flat domains came into consideration because of their relationship  with the regularity of the Poisson kernel and the harmonic measure, as shown in a series of famous and deep papers by Kenig and Toro (see for e.g. \cite{hm3,hm2,hm1,hm4}).  In particular, they are Non Tangentially Accessible (in short NTA) domains as described in \cite{JK}.  Notice that the Poisson kernel is defined as related to the solution  of the equation $ -\Delta u =0$ in $\Omega$ with $u=f$ on $\partial \Omega$. In this paper we consider the equation $(P1)$ which is  of different nature. However, our  regularity result is again based on the monotonicity formula of Alt, Caffarelli and Friedman \cite{acf}, which is known to be one of the key estimate in the study harmonic measure as well.

More recently, regularity of elliptic PDEs in Reifenberg-flat domains has been studied by Byun and Wang  in \cite{wang1,wang2,w1,w2,w3} (see also the references therein). One of their  main result regarding to equation of the type of $(P1)$ is the existence of a  global $W^{1,p}(\Omega)$ bound on  the solution. This fact will be used in  Corollary \ref{cor2} below.

The case of domains of $\mathbb R^n$ has been investigated by Caffarelli and Peral \cite{CP}. See also \cite{jklip} for the case of Lipschitz domains.  Some other type of elliptic problems in Reifenberg-flat domains can be found in \cite{LM1,LM0,L1,MT,lms}.

The present paper is the first step towards a general boundary regularity theory for elliptic PDEs in divergence form on Reifenberg-flat domains, that might be pursued in some future work. Our main result is the following.

\begin{theorem}\label{main} Let $p,q,p_0\geq 1$ be some exponents satisfying $\frac{1}{p}+\frac{1}{q}=\frac{1}{p_0}$ and $p_0>\frac{N}{2}$. Let $\alpha>0$ be any given exponent such that 
\begin{eqnarray}
\alpha< \frac{p_0-\frac{N}{2}}{p_0}. \label{defbeta4}
\end{eqnarray}
Then one can find an $\varepsilon = \varepsilon(N,\alpha)$  such that the following holds. Let $\Omega \subseteq \R^N$  be  an $(\varepsilon, r_0)$-Reifenberg-flat domain for some $r_0>0$, and let $u$ be a solution for the problem $(P1)$ in $\Omega$ with $u\in L^p(\Omega)$ and $f\in L^q(\Omega)$. Then  
$$u \in C^{0,\alpha}(B(x,r_0/12)\cap \overline{\Omega}) \quad \forall x\in \overline{\Omega}.$$
Moreover  $\|u\|_{C^{0,\alpha}(B(x,r_0/12)\cap \overline{\Omega})}\leq  C(N, r_0, \alpha,p_0,\|u\|_p,\|f\|_q)$.
\end{theorem}

Observe that in the statement of Theorem \ref{main}, some a priori  $L^p$ integrability on $u$ is needed to get some H\"older regularity. In what follows we shall see at least two situations where we know that $u\in L^p$ for some $p>2$, and consequently state two Corollaries where the integrability hypothesis is given on $f$ only, without any a priori requierement on $u$.

First, notice that when  $f\in L^q(\Omega)$ for $2\leq q\leq +\infty$, then the application $v \mapsto \int_{\Omega}vf dx$ is a bounded Linear form on $W^{1,2}_0(\Omega)$, endowed with the scalar product $\int_\Omega \nabla u \cdot \nabla v \;dx$. Therefore, using Riesz representation Theorem we deduce the existence of a unique weak solution $u \in W^{1,2}_0(\Omega)$ for the problem (P1). Moreover, the Sobolev inequality says that 
$u \in L^{2^*}(\Omega)$, with $2^*=\frac{2N}{N-2}$. Some simple computations shows that in this situation, $u$ and $f$ verify the statement of Theorem \ref{main} provided that $2\leq N\leq 5$, which leads to the following corollary.

\begin{corollary} \label{cor1} Assume that $2\leq N\leq 5 $ and  let $q \in I_N$ be given where
$$I_N=\left\{
\begin{array}{cc}
[2,+\infty) &\text{ if } N=2 \\
(\frac{2N}{6-N}, +\infty) & \text{ for } 3\leq N \leq 5.
\end{array}
\right.
$$
Then for any $\alpha>0$ verifying
$$\alpha <  1-\frac{N}{2}\left(\frac{N-2}{2N}+\frac{1}{q}\right) ,$$
we can find an $\varepsilon = \varepsilon(N,\alpha)$ such that the following holds. Let $\Omega \subseteq \R^N$  be  an $(\varepsilon, r_0)$-Reifenberg-flat domain for some $r_0>0$, let $f\in L^q(\Omega)$ and let $u \in W^{1,2}_0(\Omega)$ be the unique solution for the problem $(P1)$ in $\Omega$. Then  
$$u \in C^{0,\alpha}(B(x,r_0/12)\cap \overline{\Omega}) \quad \forall x\in \overline{\Omega}.$$
Moreover  $\|u\|_{C^{0,\alpha}(B(x,r_0/12)\cap \overline{\Omega})}\leq  C(N, r_0, \alpha,q,\|\nabla u\|_2,\|f\|_q)$.
\end{corollary}
\begin{proof}  Since $u \in W^{1,2}_0(\Omega)$, the Sobolev embedding says that 
$u \in L^{p}(\Omega)$, with $p=2^*=\frac{2N}{N-2}$. And by assumption $f\in L^q$ for some $q \in I_N$ (notice that $q\geq 2$, which guarantees existence and uniqueness of the weak solution). We now try to apply Theorem \ref{main} with those $p$ and $q$. Let $p_0$ be defined by
$$\frac{1}{p}+\frac{1}{q}=\frac{1}{p_0} .$$
Then a simple computation yields that $p_0> \frac{N}{2}$, provided that
$$q>\frac{2N}{6-N}.$$
This fixes the range of dimension $N\leq 5$ and notice that in this case $\frac{2N}{6-N}\geq 2$ except for $N=2$, which justifies the definition of $I_N$. We then conclude by applying Theorem \ref{main}. 
\end{proof}

In the proof of Corollary \ref{cor1} we brutally  used the Sobolev embedding on $W^{1,2}_0(\Omega)$ to obtain an $L^p$ integrability on $u$. But under some natural hypothesis we can get more using a Theorem by Byun and Wang \cite{w1}. Precisely, if $f={\rm div } F$ for some $F \in L^2$, then $f$ lies in the dual space of $H^1_0(\Omega)$ which guarantees the existence and uniqueness of a weak solution $u$ for (P1), again by the Reisz representation Theorem. The theorem of Byun and Wang \cite[Theorem 2.10]{w1} implies moreover that if $F\in L^r(\Omega)$ then  $\nabla u \in L^r(\Omega)$ as well. But then the Sobolev inequality says that   $u \in L^{r^*}$ which allows us to apply Theorem \ref{main} for a larger range of dimensions and exponents.  Of course this analysis is interesting only for $r\leq N$ because if $r>N$ we directly get some H\"older estimates by the classical Sobolev embedding. This leads to our second corollary.

\begin{corollary} \label{cor2} Let  $\frac{N}{3}< r\leq N$ and  $q>\frac{rN}{3r-N}$ be given, so that moreover $r\geq 2$.  Then  for any $\alpha>0$ satisfying
$$\alpha <   1-\frac{N}{2}\left(\frac{1}{r^*}+\frac{1}{q}\right), \quad \text{ with } r^*:=\frac{rN}{N-r},$$
we can find an $\varepsilon = \varepsilon(N,\alpha, |\Omega|, r)$ such that the following holds. Let $\Omega \subseteq \R^N$  be  an $(\varepsilon, r_0)$-Reifenberg-flat domain for some $r_0>0$, let $f\in L^q(\Omega)$ and assume that $f=-{\rm div } F$  for some $F\in L^r(\Omega)$. Let $u \in W^{1,2}_0(\Omega)$ be the unique weak solution for the problem $(P1)$ in $\Omega$. Then  
$$u \in C^{0,\alpha}(B(x,r_0/12)\cap \overline{\Omega}) \quad \forall x\in \overline{\Omega}.$$
Moreover  $\|u\|_{C^{0,\alpha}(B(x,r_0/12)\cap \overline{\Omega})}\leq  C(N, r_0, \alpha,r,q,|\Omega|, \|f\|_q , \|F\|_r)$.
\end{corollary}
\begin{proof} First we apply \cite[Theorem 2.10]{w1} which provides the existence of a threshold  $\varepsilon_0=\varepsilon(N,|\Omega|,r)$ such that  for any solution $u \in W^{1,2}_0(\Omega)$ of (P1) with $f=-{\rm div } F$ and $F \in L^r(\Omega)$,  we have that $\nabla u \in L^r(\Omega)$, provided that $\Omega$ is $(\varepsilon_0,r_0)$-Reifenberg-flat. But then the Sobolev inequality implies that $u\in L^{r^*}$ with 
$$r^*=\frac{rN}{N-r} \quad \text{ if } r<N,$$
and $r^*=+\infty$ otherwise.

In order to apply Theorem \ref{main} we define $p_0$ such that 
$$\frac{1}{r^*} +\frac{1}{q}=\frac{1}{p_0},$$
and we only need to check that $p_0<\frac{N}{2}$. This implies the following condition on $r$ and $q$ :
$$r>\frac{N}{3} \text{ and } q>\frac{rN}{3r-N},$$
as required in the statement of the Corollary. We finally conclude by applying Theorem \ref{main}.
\end{proof}

Our  approach to prove Theorem \ref{main} follows the one that was already used in \cite{lms} to control the energy of eigenfunctions near the boundary of Reifenberg-flat domains, and that we apply here to other PDE than the eigenvalue problem. The main ingredient in the proof is a variant of   Alt-Caffarelli-Friedman's monotonicity formula \cite[Lemma 5.1]{acf}, to control the behavior of the Dirichlet energy in balls centered at the boundary, as well as in the interior of $\Omega$. Then we conclude by Morey-Campanato Theorem.

%%%%%%%%%%%%%%%%%%%%%%%%%%%%

\section{The monotonicity Lemma}

We begin with a technical Lemma which basically contains the justification of an integration by parts. The proof is exactly the same as the first step of \cite[Lemma 15]{lms} but we decided to provide here the full details for the convenience of the reader.

\begin{lemma} \label{ipp}Let $u$ be a solution for the problem $(P1)$. Then  for every $x_0 \in \partial \Omega$ and  a.e. $r>0$ we have
\begin{eqnarray}
\int_{B(x_0,r)\cap \Omega}&|\nabla u|^2& |x|^{2-N}dx \leq  r^{2-N}\int_{\partial B(x_0,r)\cap \Omega} u \frac{\partial u}{\partial \nu}dS \notag \\
&+& \frac{(N-2)}{2}r^{1-N}\int_{\partial B(x_0,r)\cap \Omega} u^2dS   + \; \int_{B(x_0,r)\cap \Omega}u f |x|^{2-N} dx. \label{mono00}
\end{eqnarray}
\end{lemma}

\begin{proof}  Although~\eqref{mono00} can be formally obtained through an integration by parts, the rigorous proof is a bit technical. In the sequel we use the notation $\Omega_r^+:=B(x_0,r)\cap \Omega$, and $S_r^+:=\partial B(x_0,r) \cap \Omega$. We find it convenient to define, for a given $\varepsilon>0$, the regularized norm
$$|x|_{\varepsilon}:= \sqrt{x_1^2+x_2^2+\dots +x_N^2 +\varepsilon},$$
so that $|x|_{\varepsilon}$ is a $C^\infty$ function.  A direct computation shows that
$$\Delta(|x|_{\varepsilon}^{2-N})=(2-N)N\frac{\varepsilon}{|x|_{\varepsilon}^{N+2}}\leq 0,$$ 
in other words $|x|_{\varepsilon}^{2-N}$ is superharmonic, and hopefully enough this goes in the right direction regarding to the next inequalities.

We use one more regularization thus we let $u_n \in C^\infty_c(\Omega)$ be a sequence of functions converging in $W^{1,2}(\R^N)$ to $u$. We now proceed as in the proof of Alt, Caffarelli and Friedman monotonicity formula \cite{acf}: by using the equality
\begin{eqnarray}
\Delta(u_n^2)=2|\nabla u_n|^2 + 2 u_n \Delta u_n  \label{mono1}
\end{eqnarray}
 we deduce that
\begin{eqnarray}
2\int_{\Omega^+_r}|\nabla u_n|^2 |x|_\varepsilon^{2-N} = \int_{\Omega^+_r} \Delta(u_n^2) |x|_\varepsilon^{2-N} - 2\int_{\Omega^+_r}(u_n \Delta u_n) |x|_\varepsilon^{2-N} . \label{mono2}
\end{eqnarray}
Since $\Delta( |x|_\varepsilon^{2-N})\leq 0$,  the Gauss-Green Formula yields
\begin{eqnarray}
\int_{\Omega^+_r} \Delta(u_n^2) |x|_\varepsilon^{2-N}dx  = \int_{\Omega^+_r} u_n^2 \Delta( |x|_\varepsilon^{2-N})dx + I_{n,\varepsilon} (r) \leq I_{n,\varepsilon} (r) \;\label{mono3},
\end{eqnarray}
where 
$$I_{n,\varepsilon}(r)=(r^2+\varepsilon)^{\frac{2-N}{2}}\int_{\partial \Omega^+_r} 2u_n \frac{\partial u_n}{\partial \nu}dS +(N-2)\frac{r}{(r^2+\varepsilon)^{\frac{N}{2}}}\int_{\partial \Omega^+_r} u_n^2dS.$$ 
In other words, \eqref{mono2} reads
\begin{eqnarray}
2\int_{\Omega^+_r}|\nabla u_n|^2 |x|_\varepsilon^{2-N}dx \leq I_{n,\varepsilon} (r)- 2\int_{\Omega^+_r}(u_n \Delta u_n) |x|_\varepsilon^{2-N} dx. \label{mono3bis}
\end{eqnarray}
We now want to pass to the limit, first as  $n\to +\infty$, and then as $\varepsilon\to 0^+$. To tackle some technical problems, we first integrate over $r \in [r,r+\delta]$ and divide by $\delta$, thus obtaining
\begin{eqnarray}
\frac{2}{\delta}\int_{r}^{r+\delta}\left(\int_{\Omega^+_{\rho}}|\nabla u_n|^2 |x|_\varepsilon^{2-N}dx\right) d \rho\leq  A_n-R_n, \label{mono3ter}
\end{eqnarray}
where 
$$A_n=\frac{1}{\delta}\int_{r}^{r+\delta}I_{n,\varepsilon}(\rho) d\rho$$
and 
$$R_n= 2\frac{1}{\delta}\int_{r}^{r+\delta}\left(\int_{\Omega^+_\rho}(u_n \Delta u_n) |x|_{\varepsilon}^{2-N} dx\right) d\rho.$$
First, we investigate the limit of $A_n$ as $n \to + \infty$: by applying the coarea formula, we rewrite $A_n$ as
$$
    A_n = \frac{1}{\delta} \Bigg( 2 \int_{\Omega^+_{r + \delta} \setminus \Omega^+_r}   (|x|^2+\varepsilon)^{\frac{2-N}{2}} u_n \, \nabla u_n \cdot x \, dx  +
    (N-2) \int_{\Omega^+_{r + \delta} \setminus \Omega^+_r}  \frac{|x|}{(|x|^2+\varepsilon)^{\frac{N}{2}}  } u_n^2 dx \Bigg). 
$$
Since $u_n$ converges to $u$ in $W^{1,2}(\R^N)$ when $n\to +\infty$, then by using again the coarea formula we get that 
$$A_n \longrightarrow \frac{1}{\delta}\int_{r}^{r+\delta}I_{\varepsilon}(\rho) d\rho \quad  \qquad n \to + \infty, $$
where
$$I_{\varepsilon} (\rho)=(\rho^2+\varepsilon)^{\frac{2-N}{2}}\int_{\partial \Omega^+_{\rho}} 2u \frac{\partial u}{\partial \nu}dS +(N-2)\frac{\rho}{(\rho^2+\varepsilon)^{\frac{N}{2}}}\int_{\partial \Omega^+_{\rho}} u^2dS.$$
Next, we investigate the limit of $R_n$ as $n \to + \infty$. By using Fubini's Theorem, we can rewrite $R_n$ as
$$R_n=\int_{\Omega}(u_n\Delta u_n)G(x)dx,$$ 
where $$G(x)=|x|_\varepsilon^{2-N}\frac{1}{\delta}\int_{r}^{r+\delta}{\bf 1}_{\Omega^+_\rho}(x)d\rho.$$
Since 
$$
\frac{1}{\delta}\int_{r}^{r+\delta}{\bf 1}_{\Omega^+_\rho}(x)d\rho = 
\left\{
\begin{array}{ll}
1 &\text{ if } x \in \Omega^+_{r}\\
\displaystyle{\frac{r+\delta-|x|}{\delta}} & \text{ if } x \in \Omega^+_{r+\delta}\setminus \Omega^+_{r}\\
0 &\text{ if } x \notin \Omega^+_{r+\delta}\\
\end{array}
\right.,
$$
then $G$ is Lipschitz continuous and hence by recalling $u_n \in C^{\infty}_c (\Omega)$ we get  
\begin{equation*}
\begin{split}
&      \left| \int_{\Omega} \Big( u_n \Delta u_n - u \Delta u \Big) G dx     \right| \leq 
         \left| \int_{\Omega} \Big( \Delta u_n -  \Delta u \Big) u_n G dx     \right| +
          \left| \int_{\Omega} \Big( u_n - u \Big) \Delta u  G dx     \right| \\ 
 &    =  \left| \int_{\Omega} \Big( \nabla u_n -  \nabla u \Big)  \Big( u_n \nabla G+ \nabla u_n G\Big) dx     \right|     
      +
      \left| \int_{\Omega}\nabla u \Big( ( \nabla u_n - \nabla u )    G+ ( u_n - u ) \nabla G \Big)dx     \right|  \\
 & \leq      \|  \nabla u_n -  \nabla u \|_{L^2 (\Omega)}  
   \Big( \| u_n \|_{L^2 (\Omega)} \| \nabla G\|_{L^{\infty} (\Omega)} + 
  \| \nabla u_n \|_{L^{2} (\Omega)} \|G\|_{L^{\infty} (\Omega) } \Big)  \phantom{\int} \\ 
  & \quad + 
  \| \nabla u  \|_{L^2(\Omega)} \Big(  \|  \nabla u_n -  \nabla u \|_{L^2 (\Omega)}    \|G \|_{L^{\infty} (\Omega)} +
  \|  u_n -   u \|_{L^2 (\Omega)}    \|\nabla G \|_{L^{\infty} (\Omega)}\Big)  \phantom{\int} \\
\end{split}
\end{equation*}
and hence the expression at the first line converges to $0$ as $n \to + \infty$.

By combining the previous observations and by recalling that $u$ satisfies the equation in the problem $(P1)$ we infer that by passing to the limit $n \to \infty$ in~\eqref{mono3ter} we get 
$$
   \frac{2}{\delta}\int_{r}^{r+\delta}\left(\int_{\Omega^+_{\rho}}|\nabla u|^2 |x|_\varepsilon^{2-N}dx\right) dr
   \leq \frac{1}{\delta}\int_{r}^{r+\delta}I_{\varepsilon}(\rho) d\rho + \frac{2}{\delta}\int_{r}^{r+\delta}\left(\int_{\Omega^+_\rho} u f |x|_{\varepsilon}^{2-N} dx\right) d\rho.
$$
Finally, dividing by 2, by passing to the limit $\delta \to 0^+$, and then $\varepsilon \to 0^+$ we obtain~\eqref{mono00}. 
\end{proof}
%%%%%%%%%%%%%%%%%%%%%%%%%%

Next, we will need the following Lemma of Gronwall type.

\begin{lemma} \label{Gronwall}Let $\gamma>0$, $r_0>0$, $\psi : (0,r_0) \to \R$ be a continuous function and $\varphi : (0,r_0)  \to \R$ be an absolutely continuous function that satisfies the following inequality  for a.e. $r \in (0,r_0)$,
\begin{eqnarray}
\varphi(r)\leq \gamma r \varphi'(r) + \psi(r). \label{gronwall1}
\end{eqnarray}
Then 
$$r\mapsto \frac{\varphi(r)}{r^{\frac{1}{\gamma}}}+\frac{1}{\gamma}\int_0^r \frac{\psi(s)}{s^{1+\frac{1}{\gamma}}}ds $$
is a nondecreasing function on $(0,r_0)$.
\end{lemma}

\begin{proof} We can assume that
$$\int_0^{r_0} \frac{\psi(s)}{s^{1+\frac{1}{\gamma}}}ds<+\infty,$$
otherwise the Lemma is trivial. Under our hypothesis, the function 
$$F(r):=\frac{\varphi(r)}{r^{\frac{1}{\gamma}}}+\frac{1}{\gamma}\int_0^r \frac{\psi(s)}{s^{1+\frac{1}{\gamma}}}ds$$
is differentiable a.e. and absolutely continuous. A computation gives 
\begin{eqnarray}
F'(r)&=&\frac{\varphi'(r)r^{\frac{1}{\gamma}}-\varphi(r)\frac{1}{\gamma}r^{\frac{1}{\gamma}-1}}{r^{\frac{2}{\gamma}}}+\frac{1}{\gamma}\frac{\psi(r)}{r^{1+\frac{1}{\gamma}}} \notag\\
&=& \frac{\varphi'(r)-\varphi(r)\frac{1}{\gamma}r^{-1}}{r^{\frac{1}{\gamma}}}+\frac{1}{\gamma}\frac{\psi(r)}{r^{1+\frac{1}{\gamma}}} \notag
\end{eqnarray}
thus  \eqref{gronwall1} yields
$$r\gamma F'(r)=\frac{r\gamma \varphi'(r)-\varphi(r)+\psi(r)}{r^{\frac{1}{\gamma}}}\geq 0,$$
which implies that $F$ is nondecresing.
\end{proof}

We now prove the monotonicity Lemma, which is inspired by  Alt, Caffarelli and Friedman \cite[Lemma 5.1.]{acf}. The following  statement and its proof, is  an easy variant of \cite[Lemma 15]{lms}, where the same estimate is performed on Dirichlet eigenfunctions of the Laplace operator. We decided to write the full details in order to enlighten the role of the second member $f$ in the inequalities.

\begin{lemma} \label{monot} Let $\Omega \subseteq \R^N$ be a bounded domain and let $u$ be a solution for the problem $(P1)$.  Given  $x_0 \in \overline{\Omega}$ and a radius $r>0$, we denote by $\Omega_r^+:=B(x_0,r)\cap \Omega$, by $S_r^+:=\partial B(x_0,r) \cap \Omega$ and by $\sigma(r)$ the first Dirichlet eigenvalue of the Laplace operator on the spherical domain $S_r^+$. If there are constants $r_0>0$ and $\sigma^{\ast} \in ]0, N-1[$ such that
\begin{eqnarray}
    \inf_{0< r < r_0} (r^{2}\sigma(r)) \ge \sigma^*, \label{hyp1}
\end{eqnarray} 
then the function 
\begin{equation}
\label{e:f}r\mapsto \Bigg(\frac{1}{r^\beta}\int_{\Omega_r^+} \frac{|\nabla u |^2}{|x-x_0|^{N-2}} \;dx \Bigg)+\beta\int_0^r \frac{\psi(s)}{s^{1+\beta}}ds
\end{equation}
is non decreasing on $]0,r_0[$, where $\beta \in ]0, 2[$ is given by
$$
     \beta=\sqrt{(N-2)^2 +4\sigma^*}-(N-2) 
$$
and 
$$\psi(s):=\int_{\Omega_s^+}|u f| |x-x_0|^{2-N}dx.$$
We also have the bound
\begin{eqnarray}
         \int_{\Omega_{r_0/2}^+} \frac{|\nabla u |^2}{|x-x_0|^{N-2}} \;dx \leq   
          C (N,r_0,\beta) \|\nabla u\|_{L^2(\Omega)}^2+\psi(r_0).
                       \label{boundM}
\end{eqnarray}
\end{lemma}

\begin{remark} Of course the Lemma is interesting only when 
\begin{eqnarray}
\int_0^{r_0} \frac{\psi(s)}{s^{1+\beta}}ds<+\infty. \label{hyppp}
\end{eqnarray}
This will be satisfied if $u$ and $f$ are in some $L^p$ spaces with  suitable exponents, as will be shown in Lemma \ref{computation}.
\end{remark}

\begin{remark} Notice that it is not known in general whether $\nabla u \in L^p(\Omega)$ for some $p>2$ and therefore  it is not obvious to find a bound for the left hand side of \eqref{boundM}.
\end{remark}

\begin{proof} 
We assume without lose of generality that $x_0=0$ and to simplify notation we denote by $B_r$ the ball $B(x, r)$.

By Lemma \ref{ipp} we know that for a.e. $r>0$,
\begin{eqnarray}
\int_{\Omega_r^+}|\nabla u|^2 |x|^{2-N}dx &\leq & r^{2-N}\int_{S_r^+} 2u \frac{\partial u}{\partial \nu}dS +\frac{(N-2)}{2}r^{1-N}\int_{S^+_r} u^2dS \notag \\
&\quad & \quad \quad + \int_{\Omega^+_r}u f |x|^{2-N} dx. \label{mono00}
\end{eqnarray}

Let us define 
\begin{eqnarray}
\psi(r):=\int_{\Omega_r^+}|u f| |x|^{2-N}dx  \label{mono8}
\end{eqnarray}
and assume that \eqref{hyppp} holds (otherwise there is nothing to prove).

Next, we point out that the definition of $\sigma^\ast$ implies that
\begin{eqnarray}
 \int_{S_r^+} u^2dS  \leq  \frac{1}{\sigma^*}r^2 \int_{S_r^+} |\nabla_{\tau}u |^2dS \qquad 
 r\in \, ]0,r_0[, \label{mono4}
 \end{eqnarray}
where $\nabla_{\tau}$ denotes the tangential gradient on the sphere. Also, let $\alpha >0$ be a parameter that will be fixed later, then by combining Cauchy-Schwarz inequality, \eqref{mono4} and the inequality $\displaystyle{ab\leq \frac{\alpha} {2}a^2 + \frac{1}{2 \alpha} b^2}$, we get
\begin{eqnarray} 
\Bigg|\int_{S_r^+} u \frac{\partial u}{\partial \nu}dS \Bigg|&\leq& \Bigg(\int_{S^+_r} u^2 dS\Bigg)^\frac{1}{2} \Bigg(\int_{S^+_r} |\frac{\partial u}{\partial \nu}|^2dS \Bigg)^\frac{1}{2} \notag \\
&\leq &  \frac{r}{\sqrt{\sigma^*}} \Bigg(\int_{S^+_r} |\nabla_\tau u|^2dS \Bigg)^\frac{1}{2} \Bigg(\int_{S^+_r} |\frac{\partial u}{\partial \nu}|^2dS \Bigg)^\frac{1}{2} \notag \\
&\leq &  \frac{r}{\sqrt{\sigma^*}}  \Bigg( \frac{\alpha}{2}\int_{S^+_r} |\nabla_\tau u|^2dS + \frac{1}{2\alpha}\int_{S^+_r} |\frac{\partial u}{\partial \nu}|^2dS \Bigg). \label{mono5}
\end{eqnarray}
Hence, 
\begin{equation}
\label{e:rhs}
\begin{split}
          r^{2-N}\int_{S_r^+} 2u \frac{\partial u}{\partial \nu}dS +(N-2)r^{1-N}\int_{S^+_r} u^2dS  & \leq 
        r^{2-N} \frac{2 r }{\sqrt{\sigma^\ast}} \left[ \frac{\alpha}{2}  
        \int_{S_r^+}  |\nabla_{\tau} u |^2 dS+ 
        \frac{1}{2\alpha}\int_{S^+_r} |\frac{\partial u}{\partial \nu}|^2dS    \right]   \\
        & \quad + (N-2) r^{1-N} \frac{1}{\sigma^\ast} r^2
         \int_{S_r^+}  |\nabla_{\tau} u |^2 dS \\
\end{split}
\end{equation}
 $$\leq r^{3-N} \left[ \left( \frac{\alpha}{\sqrt{\sigma^{\ast}}}  + \frac{N-2}{\sigma^\ast }\right) \int_{S_r^+}  |\nabla_{\tau} u |^2 dS +\frac{1}{\alpha \sqrt{\sigma^\ast}} 
         \int_{S_r^+} |\frac{\partial u}{\partial \nu}|^2 dS   \right]. $$

Next, we choose $\alpha>0$ in such a way that
$$\frac{\alpha}{\sqrt{\sigma^*}} + \frac{N-2}{\sigma^*}=\frac{1}{\alpha\sqrt{\sigma^*}},$$
namely 
$$\alpha = \frac{1}{2\sqrt{\sigma^*}}\big[\sqrt{(N-2)^2 +4\sigma^*}-(N-2) \big].$$
Hence, by combining~\eqref{mono00},~\eqref{mono4} and~\eqref{e:rhs}  
we finally get
\begin{eqnarray}
\int_{\Omega^+_r}|\nabla u|^2 |x|^{2-N}dx \leq  r^{3-N} \gamma(N,\sigma^*) \int_{S_r^+}|\nabla u|^2dS +  \psi(r) \label{mono9},
\end{eqnarray}
where
$$\gamma(N,\sigma^*)=\Big[\sqrt{(N-2)^2 +4\sigma^*}-(N-2)\Big]^{-1}.$$

Let us set 
$$
    \varphi(r):=\int_{\Omega_r^+}|\nabla u|^2 |x|^{2-N}dx
$$ 
and  observe that 
$$
     \varphi'(r)=r^{2-N} \int_{S_r^+}|\nabla u|^2
     \quad a.e. \, r\in \, ]0,r_0[\, , 
$$  
hence~\eqref{mono9} implies that
\begin{eqnarray}
\varphi(r) \leq    \gamma r \varphi'(r) +  \psi(r),\label{mono10}
\end{eqnarray}
with $\gamma=\gamma(N,\sigma^*)$.

But now Lemma \ref{Gronwall} exactly says that  the function 
\begin{equation}
\label{e:f}r\mapsto \frac{1}{r^\beta}\int_{\Omega_r^+} \frac{|\nabla u |^2}{|x-x_0|^{N-2}} \;dx+\beta\int_0^r \frac{\psi(s)}{s^{1+\beta}}ds
\end{equation}
is non decreasing on $(0,r_0)$, where $\beta \in (0, 2)$ is given by
$$
     \beta=\frac{1}{\gamma}=\sqrt{(N-2)^2 +4\sigma^*}-(N-2),
$$
which proves the monotonicity result.

To finish the proof of the Lemma it remains to establish ~\eqref{boundM}.  For this purpose, we start by finding a radius $r_1\in (r_0/2,r_0)$ such that  $\int_{S_{r_1}^+} |\nabla u|^2 dS$ is less than average, which means 
$$\int_{S_{r_1}^+} |\nabla u|^2 dS\leq \frac{2}{r_0} \int_{\Omega_{r_0}^+}|\nabla u|^2\;dx\leq \frac{2}{r_0}\|\nabla u\|^2_{L^2(\Omega)}.$$
By combining~\eqref{mono9} with the fact that $r_0/2\leq r_1\leq r_0$ we infer that
\begin{eqnarray}
\int_{\Omega_{r_1}^+}|\nabla u|^2 |x|^{2-N}dx &\leq& C (N,r_0,\beta) \|\nabla u\|^2_{L^2(\Omega)}+  \psi(r_1) \notag \\
&\leq & C (N,r_0,\beta) \|\nabla u\|^2_{L^2(\Omega)}+  \psi(r_0).
\end{eqnarray}
It follows that
\begin{eqnarray}
\int_{\Omega_{r_0/2}^+} \frac{|\nabla u |^2}{|x|^{N-2}} \;dx  &\leq& \int_{\Omega_{r_1}^+} \frac{|\nabla u |^2}{|x|^{N-2}} \;dx  \notag \\
 &\leq & C (N,r_0,\beta) \|\nabla u\|_{L^2(\Omega)}^2+\psi(r_0),\notag 
 \end{eqnarray}
and \eqref{boundM} is proved.
\end{proof}

%%%%%%%%%%%%%%%%%%%%%%%%%%%%%%%%%%%
\section{An elementary computation}
%%%%%%%%%%%%%%%%%%%%%%%%%%%%%%%%%%%

In order to apply Lemma  \ref{monot}, the first thing to check is that \eqref{hyppp} holds. The purpose of the following Lemma is to prove that it is the case when $u$ and $f$ are in suitable $L^p$ spaces.

\begin{lemma}\label{computation} Let $\Omega\subset \R^N$ be an arbitrary domain and let  $p_0>\frac{N}{2}$. Then for any $g \in L^{p_0}(\Omega)$, denoting
$$
\psi(r):=\int_{B(0,r)}|g| |x|^{2-N}dx ,\label{mono8Lem}
$$
we have 
\begin{eqnarray}
|\psi(r)|\leq C(N,p_0)\|g\|_{p_0} r^{\frac{2p_0-N}{p_0}} .\label{decayLem}
\end{eqnarray}
As a consequence 
$$
\int_0^{r_0} \frac{\psi(s)}{s^{1+\beta}}ds<+\infty 
$$
for any $\beta>0$ satisfying 
\begin{eqnarray}
\beta< \frac{2p_0-N}{p_0}. \label{defbetaLem}
\end{eqnarray}
\end{lemma}
\begin{proof} First, we observe that  $|x|^{2-N}\in L^m$ for any $m$ that satisfies
$$(N-2)m< N \Rightarrow  m < \frac{N}{N-2}. $$
Moreover a computation gives that under this condition,
\begin{eqnarray}
\||x|^{2-N}\|_{L^m(B(0,r))}^m=\int_{B(0,r)}|x|^{(2-N)m}= C(N,m) r^{N-m(N-2)}. \label{computeLem}
\end{eqnarray}

Let us define $m$ as being the conjugate exponent of $p_0$, namely
$$\frac{1}{m}:=1-\frac{1}{p_0}.$$
Then the hypothesis $p_0>\frac{N}{2}$ implies that $m < \frac{N}{N-2}$, and by use of  H\"older inequality and by \eqref{computeLem} we can estimate
$$|\psi(r)|\leq  \|g\|_{p_0} \|  |x|^{2-N}\|_{L^m(B(0,r))}\leq C(N,m) \|g\|_{p_0} r^{\frac{N-m(N-2)}{m}},$$
or equivalently in terms of $p_0$,
\begin{eqnarray}
|\psi(r)|\leq C(N,p_0)\|g\|_{p_0} r^{\frac{2p_0-N}{p_0}} .\label{decay0Lem}
\end{eqnarray}
The conclusion of the Lemma follows directly from \eqref{decay0Lem}. 
\end{proof}

%%%%%%%%%%%%%%%%%%%%%%%%%%%%%%%%%%%
\section{Interior estimate}
%%%%%%%%%%%%%%%%%%%%%%%%%%%%%%%%%%%

We come now to our first application of Lemma \ref{monot}, which  is a decay estimate for the energy at interior points, for arbitrary domains.

\begin{proposition}
\label{propdecayDir} 
Let $p,q,p_0>1$ be some exponents satisfying $\frac{1}{p}+\frac{1}{q}=\frac{1}{p_0}$ and $p_0>\frac{N}{2}$. Let $\beta>0$ be any given exponent such that 
\begin{eqnarray}
\beta< \frac{2p_0-N}{p_0}. \label{defbeta00}
\end{eqnarray}
Let $\Omega \subseteq \R^N$  be  any domain,  $x_0 \in \Omega$, and let $u$ be a solution for the problem $(P1)$ in $\Omega$ with $u\in L^p(\Omega)$ and $f\in L^q(\Omega)$. Then  
\begin{eqnarray}
\label{e:bdecay1}
           \int_{B(x_0,r)\cap \Omega}|\nabla u|^2 d x \leq
            C r^{N-2+\beta} \|u\|_p \|f\|_q \; \quad \forall \; r \in (0,dist(x_0,\partial \Omega)/2) \, , \label{dec000}
\end{eqnarray}
with $C= C(N,  \beta,p_0, dist(x_0,\partial \Omega))$.
\end{proposition}

\begin{proof} We assume that $x_0 =0$ and we define $r_1:=dist(x_0,\partial \Omega)$ in such a way that  $B(x,r)$ is totally contained inside $\Omega$ for $r\leq r_1$ and  under the notation of Lemma \ref{monot}, $\Omega_r^+=B(0,r)$ and by $S_r^+=\partial B(0,r)$ for any $r\leq r_1$. Under our hypothesis and in virtue of Lemma \ref{computation} that we apply with $g=uf$, we know that   \eqref{hyppp} holds for any $\beta>0$ that satisfies
\begin{eqnarray}
\beta< \frac{2p_0-N}{p_0}. \label{defbeta1}
\end{eqnarray}
Notice that $ \frac{2p_0-N}{p_0}>0$, because $p_0> N/2$. 

In the sequel we chose any exponent   $\beta>0$ satisfying \eqref{defbeta1}, so that \eqref{hyppp} holds and moreover Lemma \ref{computation} says that 
\begin{eqnarray}
\int_0^{r_1} \frac{\psi(s)}{s^{1+\beta}}ds &\leq&  C(N,p_0)\|u\|_p\|f\|_q\int_0^{r_1}s^{\frac{2p_0-N}{p_0}-1-\beta} ds \notag \\
&\leq & C(N,p_0,\beta)\|u\|_p\|f\|_q  .\label{estimm2}
\end{eqnarray}

We are now ready to prove \eqref{dec000}, $\beta$ still being a fixed exponent satisfying \eqref{defbeta1}. We recall that the first eigenvalue of the spherical Dirichlet Laplacian on the unit sphere is equal to $N-1$, thus  hypothesis \eqref{hyp1} in the present context reads
\begin{eqnarray}
    \inf_{0< r < r_1} (r^{2}\sigma(r)) =N-1, \label{heyp}
    \end{eqnarray} 
so that \eqref{hyp1} holds for any $\sigma^*< N-1$. Let us choose $\sigma^*$ exactly equal to the one that satisfies
$$
     \beta=\sqrt{(N-2)^2 +4 \sigma^*}-(N-2) .
$$
one easily verifies that $\sigma^*< N-1$ because of \eqref{defbeta1}.

 As a consequence, we are in position to apply Lemma \ref{monot} which ensures that, if $u$ is a solution for the problem $(P1)$, then the function in \eqref{e:f}
is  non decreasing. In particular, by monotonicity we know that for every $r\leq r_1$,
\begin{eqnarray}
\frac{1}{r^{N-2+\beta}}\int_{\Omega_r^+}|\nabla u |^2 \;dx&\leq& \Bigg(\frac{1}{r^\beta}\int_{\Omega_r^+} \frac{|\nabla u |^2}{|x-x_0|^{N-2}} \;dx \Bigg)+ \beta \int_0^{r} \frac{\psi(s)}{s^{1+\beta}}ds \notag\\
&\leq&  \Bigg(\frac{1}{(r_1/2)^\beta}\int_{\Omega_{r_1/2}^+} \frac{|\nabla u |^2}{|x-x_0|^{N-2}} \;dx \Bigg)+ \beta\int_0^{r_1/2} \frac{\psi(s)}{s^{1+\beta}}ds,
\end{eqnarray}
and we conclude that for every $r\leq r_1/2$,
$$\int_{\Omega_r^+}|\nabla u |^2 \;dx\leq K r^{N-2+\beta},$$
with 
$$K=\Bigg(\frac{1}{(r_1/2)^\beta}\int_{\Omega_{r_1/2}^+} \frac{|\nabla u |^2}{|x-x_0|^{N-2}} \;dx \Bigg)+ \beta\int_0^{r_1/2} \frac{\psi(s)}{s^{1+\beta}}ds.$$
Let us now provide an estimate on $K$.  To estimate the first term in $K$ we use \eqref{boundM}  to write
\begin{eqnarray}
         \int_{\Omega_{r_1/2}^+} \frac{|\nabla u |^2}{|x-x_0|^{N-2}} \;dx &\leq&   
          C (N,r_1,\beta) \|\nabla u\|_{L^2(\Omega)}^2+\psi(r_1)     \label{boundMM1} 
               \end{eqnarray}
  Then we use  \eqref{decayLem} to estimate 
  $$\psi(r_1) \leq C(N,p_0,r_1)\|u\|_p \|f\|_q,$$
  and from the equation satisfied by $u$ we get
  $$\|\nabla u\|_{L^2(\Omega)}^2= \int_{\Omega}uf dx \leq \|u\|_p \|f\|_q$$
 so that in total we have 
 \begin{eqnarray}
         \int_{\Omega_{r_1/2}^+} \frac{|\nabla u |^2}{|x-x_0|^{N-2}} \;dx &\leq&   
          C(N,r_0,\beta,p_0) \|u\|_p\|f\|_q . \notag
                        \end{eqnarray}  
Finally, the last estimate together with \eqref{estimm2} yields
$$K\leq C(N,r_1,\beta,p_0) \|u\|_p\|f\|_q,$$
and this ends the proof of the Proposition.
\end{proof}

%%%%%%%%%%%%%%%%%%%%%%%%%%%%%%%%%%%%%%

\section{Boundary estimate}

We now use Lemma \ref{monot} again to provide an estimate on the energy at  boundary points, this time for Reifenberg-flat domains.

\begin{proposition}
\label{propdecayDir2} 
Let $p,q,p_0>1$ be some exponents satisfying $\frac{1}{p}+\frac{1}{q}=\frac{1}{p_0}$ and $p_0>\frac{N}{2}$. Let $\beta>0$ be any given exponent such that 
\begin{eqnarray}
\beta< \frac{2p_0-N}{p_0}. \label{defbeta0}
\end{eqnarray}
Then one can find an $\varepsilon = \varepsilon(N,\beta)$  such that the following holds. Let $\Omega \subseteq \R^N$  be  any $(\varepsilon, r_0)$-Reifenberg-flat domain for some $r_0>0$, let  $x_0 \in \partial \Omega$ and let $u$ be a solution for the problem $(P1)$ in $\Omega$ with $u\in L^p(\Omega)$ and $f\in L^q(\Omega)$. Then  
\begin{eqnarray}
\label{e:bdecay1}
           \int_{B(x_0,r)\cap \Omega}|\nabla u|^2 d x \leq
            C r^{N-2+\beta} \|u\|_p \|f\|_q \; \quad \forall \; r \in (0,r_0/2) \, , \label{dec00}
\end{eqnarray}
with $C= C(N, r_0, \beta,p_0)$.
\end{proposition}
\begin{proof} As before we assume that  $x_0 =0$ and we denote by $\Omega_r^+:=B(0,r)\cap \Omega$ and by $S_r^+:=\partial B(0,r) \cap \Omega$. To obtain the decay estimate on  $\int_{\Omega_r^+} |\nabla u|^2 d x$ we will follow the proof of Proposition \ref{propdecayDir} : the main difference is that for boundary points, \eqref{heyp} does not hold. This is where Reifenberg-flatness will play a role.

Let  $\beta>0$, be an exponent satisfying \eqref{defbeta0}, so that invoquing Lemma  \ref{computation} we have 
\begin{eqnarray}
\int_0^{r_0} \frac{\psi(s)}{s^{1+\beta}}ds &\leq&   C(N,p_0,\beta)\|u\|_p\|f\|_q  <+\infty.\label{estimm}
\end{eqnarray}

Next, we recall that the first eigenvalue of the spherical Dirichlet Laplacian on a half sphere is equal to $N-1$ (as for the total sphere). For $t \in \, (-1,1)$, let $S_t$ be the spherical cap $S_t:= \partial B(0,1) \cap \{x_N >t\}$ so that $t=0$ corresponds to a half sphere. Let $\lambda_1(S_t)$ be the first Dirichlet eigenvalue in $S_t$. In particular, $t\mapsto \lambda_1(S_t)$ is continuous and monotone in $t$. Therefore, since {and $\lambda_1 (S_t) \to 0$ as $t \downarrow -1$, there  is $t^*(\beta)<0$ such that 
$$
     \beta=\sqrt{(N-2)^2 +4 \lambda_1(t^*)}-(N-2) 
$$

By applying  the definition of Reifenerg flat domain, we infer that, if ${\varepsilon < t^* (\eta) /2}$, then $\partial B(x_0,r) \cap \Omega$ is contained in a spherical cap homothetic to $S_{t^*}$ for every $r \leq r_0$. 
Since the eigenvalues scale of by factor $r^2$ when the domain expands of a factor $1/r$,} by the monotonicity property of the eigenvalues with respect to domains inclusion, we have 
\begin{equation}
\label{e:inf}
\inf_{r<r_0} r^{2}\lambda_1(\partial B(x_0,r) \cap \Omega) \geq \lambda_1(S_{t^*})= \frac{\beta}{2}\Big(\frac{\beta}{2}+N-2\Big).
\end{equation}
 As a consequence, we are in position to apply the monotonicity Lemma (Lemma \ref{monot}) which ensures that, if $u$ is a solution for the problem $(P1)$ and $x_0\in \partial \Omega$, then the function in \eqref{e:f}
is  non decreasing. We then conclude as in the proof of Proposition \ref{propdecayDir}, i.e.  by monotonicity we know that for every $r\leq r_0<1$,
\begin{eqnarray}
\frac{1}{r^{N-2+\beta}}\int_{\Omega_r^+}|\nabla u |^2 \;dx&\leq& \Bigg(\frac{1}{r^\beta}\int_{\Omega_r^+} \frac{|\nabla u |^2}{|x-x_0|^{N-2}} \;dx \Bigg)+ \beta \int_0^{r} \frac{\psi(s)}{s^{1+\beta}}ds \notag\\
&\leq&  \Bigg(\frac{1}{(r_0/2)^\beta}\int_{\Omega_{r_0/2}^+} \frac{|\nabla u |^2}{|x-x_0|^{N-2}} \;dx \Bigg)+ \beta\int_0^{r_0/2} \frac{\psi(s)}{s^{1+\beta}}ds,
\end{eqnarray}
hence for every $r\leq r_0/2$,
$$\int_{\Omega_r^+}|\nabla u |^2 \;dx\leq K r^{N-2+\beta},$$
with, 
$$K=\Bigg(\frac{1}{(r_0/2)^\beta}\int_{\Omega_{r_0/2}^+} \frac{|\nabla u |^2}{|x-x_0|^{N-2}} \;dx \Bigg)+ \beta\int_0^{r_0/2} \frac{\psi(s)}{s^{1+\beta}}ds.$$
Then we estimate $K$ exactly as in the end of the  proof of Proposition  \ref{propdecayDir}, using    \eqref{boundM},  \eqref{decayLem} and \eqref{estimm} to bound
$$K\leq C(N,r_0,\beta,p_0) \|u\|_p\|f\|_q,$$
and this ends the proof of the Proposition.
\end{proof}

\section{Global decay result}

Gathering together Proposition \ref{propdecayDir} and Proposition \ref{propdecayDir2} we deduce the following global result.

\begin{proposition}\label{global}Let $p,q,p_0>1$ be some exponents satisfying $\frac{1}{p}+\frac{1}{q}=\frac{1}{p_0}$ and $p_0>\frac{N}{2}$. Let $\beta>0$ be any given exponent such that 
\begin{eqnarray}
\beta< \frac{2p_0-N}{p_0}. \label{defbeta3}
\end{eqnarray}
Then one can find an $\varepsilon = \varepsilon(N,\beta)$  such that the following holds. Let $\Omega \subseteq \R^N$  be  any $(\varepsilon, r_0)$-Reifenberg-flat domain for some $r_0>0$, and let $u$ be a solution for the problem $(P1)$ in $\Omega$ with $u\in L^p(\Omega)$ and $f\in L^q(\Omega)$. Then  
\begin{eqnarray}
           \int_{B(x,r)\cap \Omega}|\nabla u|^2 d x \leq
            C r^{N-2+\beta} \|u\|_p \|f\|_q \; \quad \forall x \in \overline{\Omega}, \; \forall \; r \in (0,r_0/6) , \label{dec3}
\end{eqnarray}
with $C= C(N, r_0, \beta,p_0)$.
\end{proposition}
\begin{proof} By Proposition \ref{propdecayDir} and Proposition \ref{propdecayDir2}, we already know that \eqref{dec3} holds true for every $x\in \partial \Omega$, or for  points $x$ such that $dist(x,\partial \Omega)\geq r_0/3$. It remains to consider balls centered at points $x\in \Omega$ verifying 
$$dist(x,\partial \Omega)\leq r_0/3.$$
Let $x$ be such a point. Then Proposition \ref{propdecayDir} directly says that \eqref{dec3} holds for every radius $r$ such that $0<r\leq dist(x,\partial \Omega)/2$, and it remains to extend this for the radii $r$ in the range
\begin{eqnarray}
dist(x,\partial \Omega)/2 \leq r\leq r_0/6. \label{range}
\end{eqnarray}
For this purpose, let $y \in \partial \Omega$ be such that 
$$dist(x,\partial \Omega)=\|x-y\|\leq r_0/3.$$
Denoting $d(x):=dist(x,\partial \Omega)$ we observe that for the $r$ that satisfies \eqref{range} we have
\begin{eqnarray}
B(x,r)\subseteq B(y,r+d(x))\subseteq B(y,3r). \label{inclusion}
\end{eqnarray}
Then  since $y \in \partial \Omega$, Proposition \ref{propdecayDir2} says that 
\begin{eqnarray}
           \int_{B(y,r)\cap \Omega}|\nabla u|^2 d x \leq
            C r^{N-2+\beta} \|u\|_p \|f\|_q\quad \; \forall \; r \in (0,r_0/2) , \label{dec4}
\end{eqnarray}
so that \eqref{dec3} follows, up to change $C$ with $3^{N-2+\beta}C$.
\end{proof}

%%%%%%%%%%%%%%%%%%%%%%%%%%%%%%%%%%%%%%%%%%
\section{Conclusion and main result}

The classical results on Campanato Spaces can be found for instance in \cite{gia}. 
We define the space 

$$\mathcal{L}^{p,\lambda}(\Omega):=\left\{ u \in L^p(\Omega) \;;\;
 \sup_{x,\rho}\left(\rho^{-\lambda}\int_{B(x,r)\cap \Omega} |u-u_{x,r}|^p dx \right)< +\infty \right\}$$
where the supremum is taken over all $x\in \Omega$ and all $\rho\leq diam(\Omega)$, and where $u_{x,r}$ means the average of $u$ on the ball $B(x,r)$. A proof of the next result can be found in   \cite[Theorem 3.1.]{gia}.

\begin{theorem}[Campanato] \label{camp} If  $N< \lambda \leq N+p$ then
$$\mathcal{L}^{p,\lambda}(\Omega)\simeq C^{0,\alpha}(\overline{\Omega}), \quad \text{ with } \alpha=\frac{\lambda-N}{p}.$$
\end{theorem}

We can now prove our main result.

\begin{proof}[Proof of Theorem \ref{main}] Considering  $u$ as a function of $W^{1,2}(\R^N)$ by setting $0$ outside $\Omega$, and applying  Proposition \ref{global} with $\beta=2\alpha$, we obtain that
\begin{eqnarray}
           \int_{B(x,r)}|\nabla u|^2 d x \leq
            C r^{N-2+\beta} \|u\|_p \|f\|_q \; \quad \forall x \in \overline{\Omega}, \; \forall \; r \in (0,r_0/6) , \label{dec5}
\end{eqnarray}
with $C= C(N, r_0, \beta,p_0)$. Recalling now the classical Poincar\'e inequality in a ball $B(x,r)$ 
$$\int_{B(x,r)}|u-u_{x,r}| dx\leq C(N)r^{1+\frac{N}{2}}\left(\int_{B(x,r)}|\nabla u|^2dx\right)^{\frac{1}{2}},$$
we get
\begin{eqnarray}
      \int_{B(x,r)}|u-u_{x,r}|dx     \leq
            C r^{N+\frac{\beta}{2}}  \; \quad \forall x \in \overline{\Omega}, \; \forall \; r \in (0,r_0/6),\label{dec6}
\end{eqnarray}
with $C= C(N, r_0, \beta,p_0,\|u\|_p ,\|f\|_q)$. But this implies that 

$$u \in \mathcal{L}^{1,N+\frac{\beta}{2}}(B(x,r_0/12)\cap \overline{\Omega}) \quad \forall x\in \overline{\Omega}.$$
Moreover 
$\frac{\beta}{2}< \frac{p_0-\frac{N}{2}}{p_0}\leq 1$
and hence Theorem \ref{camp} says that 
$$u \in C^{0,\alpha}(B(x,r_0/12)\cap \overline{\Omega})$$
with $\alpha=\frac{\beta}{2}$, and the norm is controlled by $C= C(N, r_0, \beta,p_0,\|u\|_p ,\|f\|_q)$.
 \end{proof}

\bibliographystyle{plain}
\bibliography{biblio}

% Using BibTex is not recommended but can be handled.

\end{document}